\documentclass[twoside]{article}
\usepackage{PRIMEarxiv}
\usepackage{algorithm}
\usepackage{algpseudocode}
\usepackage{amsmath}
\usepackage{amssymb}
\usepackage{amsthm}
\usepackage{amsfonts}
\usepackage{graphicx}
\usepackage{subcaption}
\usepackage{caption} 
\usepackage{dsfont}
\usepackage{mathrsfs}
\usepackage{multicol}
\usepackage{mathtools}
\usepackage{tikz}
\usetikzlibrary{calc}
\usetikzlibrary{decorations.pathreplacing,calligraphy}
\usepackage{tikz-cd} 
\usetikzlibrary{backgrounds}
\usepackage{pst-node}
\usepackage{url}
\usepackage{hyperref}

\newtheorem{theorem}{Theorem} %% --> this is the counter 
\newtheorem{lemma}[theorem]{Lemma}

\newtheorem{definition}[theorem]{Definition}

\def\Im{{\hbox{\textbf{im}}}} % image of linear transformation 
\def\Ker{{\hbox{\textbf{ker}}}} % kernel of linear transformation
\def\field{{\mathds{K}}} % field K
\def\d{{\partial}} % boundary operator
\def\C{{\mathcal{C}_{\bullet}}} % chain complex 
\def\T{{\mathcal{T}}} % target matrix (1)
\def\M{{M}} % matching matrix (1)
\def\S{{\mathcal{S}}} % source matrix (1)
\def\D{{\mathcal{D}}} % relative boundary matrix 
\def\TT{{\tilde{\T}}} % target matrix (2) 
\def\MM{{\tilde{\M}}} % matching matrix (2)
\def\SS{{\tilde{\S}}} % source matrix (2)
\def\A{{\mathcal{A}}} % permuted target matrix 
\def\B{{\mathcal{B}}} % permuted source matrix
\def\rcb{{f_{ker}}} % relative cycle birth function
\def\rbb{{f_{im}}} % relative boundary birth function

\def\F{{F_{\bullet}}} % an arbitrary filtration F
\def\G{{G_{\bullet}}} % an arbitraty filtration G 
\def\FX{{F_{\bullet}X}} % filtration on full space, F
\def\subFX{{G_{\bullet}Y}} % filtration on subspace, G
\newcommand{\cg}{\sigma} % cells global sequence
\newcommand{\cs}{\tau} % cells subspace sequence
\newcommand{\cc}{\gamma} % cells relative cycle sequence
\newcommand{\cb}{\beta} % cells relative boundary sequence

\title{A U-match Algorithm for Persistent Relative Homology}

\author{
  Christian Lentz \\
  Tufts University \\
  Medford, MA \\
  \texttt{christian.lentz@tufts.edu} \\
  \And
  Gregory Henselman-Petrusek \\
  Pacific Northwest National Laboratory \\
  USA \\
  \texttt{gregory.roek@pnnl.gov} \\
  \And
  Lori Ziegelmeier \\
  Macalester College \\
  St. Paul, MN \\
  \texttt{lziegel1@macalester.edu} \\
}

\begin{document}
\thispagestyle{empty}
\maketitle

% ============================================  

\begin{abstract}
    A central problem in data-driven scientific inquiry is how to interpret structure in noisy, high-dimensional data. Topological data analysis (TDA) provides a solution via persistent homology, which encodes features of interest as topological holes within a filtration of data. The present work extends this framework to a related invariant which uncovers topological structure of a space relative to a subspace: persistent relative homology (PRH). We show that this invariant can be computed in a simple, highly transparent and general manner, using a two-step matrix reduction technique with worst-case time complexity comparable to ordinary persistent homology. We provide proofs demonstrating the correctness and computational complexity of this approach in addition to a performance-optimized implementation for a special case. 
\end{abstract}

\keywords{persistent homology, relative homology, topological data analysis}

% ============================================  

\section{Introduction}
\label{sec:intro}
Persistent homology captures topological structure of data and its evolution across one or multiple choices of a parameter value \cite{computingPH, barcodeGhrist, simplification, Carlsson2009TopologyAD, chomp}. Various breakthroughs in the theory and computation of persistent homology have made way for the field of \textit{Topological Data Analysis} (TDA). Presently, TDA is a growing field with increasing application and interdisciplinary potential in machine learning \cite{PHandAI}, network science \cite{networkScienceSurvery}, and robotics \cite{pathFinding} to name a few. The increasing demand for topological insight into data and the success of the persistent homology paradigm motivate deeper exploration of topological invariants as tools to extract insights from scientific data. 

However, the computation of topological invariants is a substantial bottleneck. Indeed, a grand challenge in the field of TDA is to determine which topological descriptors can be computed efficiently in a scientific setting, and how to implement these computations reliably in code. The present work seeks to address this challenge for one of the most basic topological descriptors one might apply to scientific data: persistent relative homology. Specifically, we seek to better understand the capabilities and limits of \emph{general-purpose algorithms}; those which can be applied to almost any data set, with only minimal requirements on format and structure.

\subparagraph*{Problem Statement.} 
Consider a pair of filtered topological spaces $F_0 X \subseteq \cdots \subseteq F_N X$ and $G_0 Y \subseteq \cdots \subseteq G_N Y$ such that $G_t Y \subseteq F_t X$ for any $t$. Fixing a coefficient field $\field$, our objective is to compute the barcode of a persistence module
\[
H_*(F_0X/G_0Y) \rightarrow \cdots \rightarrow H_*(F_NX/G_NY)
\]
and calculate a (relative) cycle representative for each bar in the barcode. 

\subparagraph*{Contribution.} 
Consider the classical approach to computing the barcode of the absolute persistent homology module $H_*(F_0K) \rightarrow \cdots \rightarrow H_*(F_NK)$. We form the \textit{differential} or \textit{boundary} matrix $D$ and compute a so-called  $R = DV$ decomposition \cite{vines}. The absolute barcode can then be read from the sparsity pattern of $R$, and the corresponding cycle representatives can be obtained from the columns of $R$ and $V$.  

We provide a theoretical framework to show that a similar approach can be used to compute the barcode of the PRH module. In fact, we need only perform a change of basis. We begin by computing a U-match factorization $\T\M = \D\S$, where $\D$ is a permutation of boundary matrix $D$ \cite{umatch}. We then permute columns of $\T$ and $\S$ to obtain matrices $\A$ and $\B$, respectively, and obtain a second U-match decomposition $\TT\MM=(\A^{-1}\B)\SS$. The desired barcode can be obtained from the sparsity pattern of $\MM$, and the corresponding relative cycle representatives can be obtained from the columns of the matrix $\A\TT$. This method applies to any  filtered cell complex over an arbitrary coefficient field, and allows for arbitrary filtrations $\F$ and $\G$. 

The rich structure of U-match allows our work to be straightforward and accessible, ensuring that proofs of correctness require only linear algebra and basic definitions of algebraic topology. In fact, the properties of U-match will also provide two key advantages for this method: increased flexibility in application and straightforward interpretation of generators and bounding chains for relative homology classes. We include an implementation of a performance-optimized and interpretable case in Section \ref{sec:implementation} for illustration. 

 To the best of our knowledge, the algorithm presented in this work is unique in that it combines two features: (A) $\F$ and $\G$ can be \emph{arbitrary} filtrations of a finite cell complex, and (B) the procedure can be completed in two matrix factorizations, revealing not only a barcode but also persistent relative cycle representatives.  Prior work has provided algorithms for arbitrary filtrations using five factorizations; see Previous and Related Work for details.

\subparagraph*{Previous and Related Work.} 
In \cite{morozov}, Morozov presents an algorithm to construct persistence diagrams corresponding to nested sequences of \textit{images}, \textit{kernels}, and \textit{cokernels}. The method is cubic in the number of simplices, based on the standard matrix reduction of \cite{vines}, and requires (1) simplicial complexes $K_0 \subseteq K$, (2) an injective function $f: K \rightarrow \mathbb{R}$ whose sublevel sets are subcomplexes of $K$, and (3) an injective function $g$ given by a restriction of $f$ to $K_0$. The algorithm proceeds via a series of five matrix decompositions, and can be extended to sequences of nested kernels, images, and cokernels of filtered pairs of spaces via the \textit{Mapping Cylinder Lemma} of \cite{morozov}. A streamlined presentation of this work can be found in \cite{PH-kernelsImagesCokernels}. 

Motivated by computing essential homology, the authors in \cite{extendedPersistence} introduce extended persistence, using \textit{Poincar\'e} and \textit{Lefschetz} dualities to design a matrix reduction algorithm for computing relative homology via construction of a so-called \textit{cone complex}. In \cite{dualities}, the authors introduce \textit{pointwise} and \textit{global} dualities. They show that the cone complex of \cite{extendedPersistence} is unnecessary when computing PRH, and that each of the four standard persistent (co)homology modules may be inferred from each other, and thus computed in (at most) a two-step matrix reduction. More recently, the relative Delaunay-\^Cech complex $Del\textit{\^C}(X, A)$ was introduced in \cite{relativePH} as a means of efficiently computing relative \^Cech persistent homology in low dimensional Euclidean space. Specifically, given $A \subseteq X \subseteq \mathbb{R}^d$, the authors define $Del\textit{\^C}(X, A)$ and show that the underlying filtered simplicial complex has persistent homology isomorphic to the homology of the filtered \^Cech complex $\textit{\^C}_{\bullet}(X)/  \textit{\^C}_{\bullet}(A)$.

\textit{Local homology} considers the special case $H_p(X, X \setminus x)$ for a topological space $X$ and a single point $x \in X$. Morozov's work in \cite{morozov} also describes this problem. More recently, Kerber and Söls introduced the \textit{localized bifiltration} \cite{localized} where data is filtered by scale and a second filtration parameter describes the locality of the scale filtration about a single point. In \cite{local}, the authors describe the local homology of abstract simplicial complexes with applications in the analysis of graphs and hypergraphs. In \cite{robinsonCriticalNodes}, the authors present a theoretical framework for the identification of critical nodes or bottlenecks in a network that replaces computationally expensive graph algorithms with a local homology approach. 

In the main, algorithms for persistent relative homology fall into three categories: (A) image and kernel persistence, as in \cite{PH-kernelsImagesCokernels}, which do not precisely align with relative homology but incorporate similar concepts, (B) modules of form 
$
H_*(X,F_1X) \to \cdots \to H_*(X,F_NX)
$, 
which appear in \cite{extendedPersistence, dualities}, and (C) modules of form 
$
H_*(F_0X,G_0Y) \to \cdots \to H_*(F_NX,G_NY)
$ 
where $G_tY = (F_tX) \cap Y$. The authors of \cite{PH-kernelsImagesCokernels} also consider case (C) and, to our knowledge, is the only work which also considers an arbitrary module of the form 
$
H_*(F_0X,G_0Y) \to \cdots \to H_*(F_NX,G_NY)
$ 
without placing restrictions of kind (B) or (C). This is achieved via the \textit{Mapping Cylinder Lemma}, and the approach works by applying a series of (at most) five matrix decompositions. The main contribution of the present work is an algorithm that works with equal generality, using (at most) two matrix decompositions. While \cite{dualities} also achieves a two-step matrix decomposition (at most), this falls under case (B). Additionally, computing persistent relative cycle representatives are a key focus of this work, which is not present in prevailing methods and literature.

\section{Preliminaries}
\label{sec:background} 

For additional background, we refer the reader to \cite{chomp,hatcher} for pertinent definitions of algebraic topology, \cite{crossleySimplicial} for an introduction to simplicial homology, \cite{Carlsson2009TopologyAD,computingPH} for a review of persistent homology, and \cite{DW} for an introduction to computational techniques. 

\subparagraph{Chain Complexes and Homology.}
\label{sec:background:homology} 
Persistent homology is computed via a series of refinements from raw data. First, data is encoded as a \textit{chain complex} of the form
\[
\C := \dots \text{ } C_n \text{ } \xrightarrow[]{\d_n} \text{ } C_{n-1} \text{ } \xrightarrow[]{\d_{n-1}} \text{ } \dots \text{ } \xrightarrow{\d_2} \text{ } C_1 \xrightarrow[]{\d_1} \text{ } C_{0} \text{ }  \xrightarrow[]{\d_0} 0,
\]
or a sequence of linear maps and vector spaces satisfying $\d_p\d_{p+1} = 0$. We call each $C_p$ a \textit{chain vector space} and write $C = \bigoplus_p C_p$ for the space of all chains. We refer to the homomorphisms as \textit{boundary operators}.\footnote{The term boundary operator is often used interchangeably with \textit{differential operator}.} A chain complex can be neatly described by a single linear transformation whose matrix is
\[
D \text{ = } \begin{pmatrix}
        0 & \d_1 & &  \\
        & 0 & \d_2 & & \\
        & & \ddots & \ddots & \\
        & & & 0 & \d_n \\
        & & & & 0 
\end{pmatrix},
\]
or the \textit{block boundary matrix}. It is straightforward to show that the statement $\d_p\d_{p+1} = 0$ is equivalent to the statement $\Im(\d_{p+1}) \subseteq \Ker(\d_p)$, where we refer to any member of the vector space $\Im(\d_{p+1})$ as a $p$-\textit{boundary} and any member of the vector space $\Ker(\d_p)$ as a $p$-\textit{cycle}. With this, we define the $p^{th}$ \textit{homology vector space} of the chain complex $\C$ to be the quotient
$
\centering H_p(C_\bullet) = \Ker(\d_p)/\Im(\d_{p+1}),
$
where the non-trivial equivalence classes represent groups of $p$-cycles which are not $p$-boundaries. An element $\xi$ of an equivalence class $[\xi] = \xi + \Im(\d_{p+1})$ is called a \textit{cycle representative} for that class. We write $Z := \Ker(D) \cong \oplus_p \Ker(\partial_p)$ and $B: = \Im(D) \cong \oplus_p \Im(\partial_p)$ for the total space of all cycles and boundaries, respectively, aggregated over all dimensions. We likewise write $H_*(C_\bullet): = Z/B \cong \bigoplus_p H_p(C_\bullet)$ for the total space of all homology groups.

\subparagraph{Filtered Spaces.}
\label{sec:background:filtrations} 
Our discussion will focus on chain complexes that arise from a \emph{filtration}, i.e. a nested sequence of topological cell complexes
$
\FX := F_1X \subseteq F_2X \subseteq \dots \subseteq F_{N}X
$
where $F_N X= X$. In practice, $X$ is almost always a simplicial or cubical complex. To each subcomplex $F_tX$, we associate a chain complex $\C(F_t X)$. For convenience, and where context leaves no room for confusion, we will use $F_tX$ interchangeably with $C(F_tX)$, the space of all chains in $C_\bullet(F_tX)$. The \emph{birth filtration value} of a cell $\sigma \in X$ is defined as the minimum $t$ such that $\sigma \in F_tX$; we denote this value  $b_F(\sigma) = \min\{ t : \sigma \in F_tX\}$.\footnote{In practice the sequence $F_1X \subseteq F_2X \subseteq \dots \subseteq F_{N}X$ typically corresponds to a sequence of real values $t_1 \le \cdots \le t_N$, and it is common to refer to these real numbers as the filtration values.} Given a field $\field$ and a  $p$-chain $\alpha = \sum a_i\sigma_k$ with $a_i \in \field \setminus \{0\}$ and $\sigma_k \in C_p$, we define the \emph{birth filtration value of $\alpha$} as $b_F(\alpha) := \max\{b_F(\sigma_k)\}$. 

\subparagraph*{Persistence.}
\label{sec:background:persistence}
Let $\C(X)$ be a chain complex over a filtered topological space $\FX$. Each $C_p(X)$ is generated by a standard basis of $p$-cells in the cell complex, and each linear transformation $\d_{p+1}$ maps a $(p+1)$-cell to a linear combination of $p$-cells, or a $p$-\textit{chain}, which is its boundary.\footnote{The boundary of a chain $\alpha \in C_p$ is $\d_p(\alpha)$, where differential map $\d_p$ is defined for cells of $C_p$ and extended linearly to all chains in $C_p$.} %Further, $\C(X)$ inherits the hierarchical structure of $\FX$, giving rise to a \textit{filtered chain complex}. 
Figure \ref{fig:filteredChainComplex} provides an example with a sequence of chain complexes and inclusions, called \textit{chain maps}, on the corresponding chain vector spaces of the complexes. 
\begin{figure}[H]
\begin{center}
\begin{tikzcd}
0 \arrow[r, "\partial_{n+1}"]  & C_n(F_{a}X) \arrow[r, "\partial_n"] \arrow[d, "f_n", hook] & \dots \arrow[r, "\partial_1"] & C_0(F_{a}X) \arrow[r, "\partial_0"] \arrow[d, "f_0", hook] & 0 \\
0 \arrow[r, "\partial_{n+1}"'] & C_n(F_bX) \arrow[r, "\partial_n"'] & \dots \arrow[r, "\partial_1"'] & C_0(F_bX) \arrow[r, "\partial_0"'] & 0
\end{tikzcd}
\end{center}
\caption{Two levels of a filtered chain complex, over a filtered topological space $\FX$, whose chain maps are of the form $f_p: C_p(F_{a}X) \hookrightarrow C_p(F_bX)$ for some $a \leq b$.}
\label{fig:filteredChainComplex}
\end{figure} 

These chain maps induce linear transformations on homology given by $f_p^*: H_p(F_{a}X) \rightarrow H_p(F_bX)$ for some $a \leq b$ which can be placed in a sequence
$
\centering H_*(\FX) := H_*(F_{1}X) \rightarrow \cdots \rightarrow H_*(F_{N}X),
$ 
called a (graded) \textit{persistent homology module}. A result of Gabriel \cite{Gabriel} shows that each persistent homology module over a finite cell complex admits a direct-sum decomposition of the form
$
H_*(\FX) \cong \oplus_\xi \mathds{I} \langle b_{\xi}, d_{\xi} \rangle,
$
where each $\langle b_{\xi}, d_{\xi} \rangle$ is a so-called \emph{interval module} of the form $0 \to \cdots \to 0 \to \field \xrightarrow{id} \cdots  \xrightarrow{id} \field \to 0 \to \cdots \to 0$ describing the birth $b_{\xi}$ and death $d_{\xi}$ of some cycle representative $\xi \in H_*(\FX)$. The multiset of interval modules gives the \textit{barcode} of the persistence module $H_*(\FX)$.

\section{Persistent Relative Homology}
\label{sec:prh}

Now, we extend the definitions and notation of the previous section to PRH. We direct the reader to \cite{crossleyDeconstructionist} for more on quotient topology and \cite{hatcher} for more on relative homology. We consider a filtered pair of topological spaces $(X, Y)$ equipped with filtrations $\F$ and $\G$ respectively. We denote the pair together with their filtrations as $(\FX, \subFX)$ and require it to satisfy (1) $F_NX = X$, (2) $G_NY = X$, and (3) $G_tY \subseteq F_tX$ for each $t \in \{1, \dots, N\}$ for $N \in \mathds{N}$. Note that (2) is to ensure that the filtration terminates.  

\subparagraph*{Quotient Chain Complexes.}
PRH can be considered as the homology of a specific type of filtered chain complex which arises from a sequence of \textit{relative chain vector spaces} describing the span of all \textit{relative chains} in a \textit{quotient space} $X/Y$. In dimension $p$, this chain vector space is the quotient $C_p(X,Y) = C_p(X)/C_p(Y)$. Since $\d_p(C_p(Y)) \subseteq C_{p-1}(Y)$, there is no need to define a so-called relative boundary operator. Thus, we use the differentials $\d_p$ and a sequence of chain vector spaces of the form $C_p(X,Y)$ to define a \textit{quotient chain complex} $\C(X/Y)$ over ($\FX$, $\subFX$). 

\subparagraph*{Relative Homology.}
In absolute homology, we interpret any $\alpha \in C_p(X)$ to be a $p$-cycle if and only if $\d_p(\alpha)$ is trivial. A \textit{relative p-cycle} is defined by taking any chain of $C_{p-1}(Y)$ to be trivial. Similarly, there is a modified notion of a boundary in the relative homology setting. 
\begin{definition}\label{def:relCycleAndBoundary}
    A \textbf{relative p-cycle} is an element of $\partial_p^{-1}(C_{p-1}(Y)) = \{\alpha \in C_p(X): \partial_p(\alpha) \in C_{p-1}(Y)\}$. Across all $p$, we denote the vector space of relative cycles as $Z(X,Y) = C(X) \cap \partial^{-1}(Y)$. Similarly, a \textbf{relative p-boundary} is any chain $\alpha = \d_{p+1} (\beta) + \gamma$ for some $\beta \in C_{p+1}(X)$ and $\gamma \in C_{p}(Y)$ and the vector space of relative boundaries is $B(X,Y) = B(X) + C(Y)$.
\end{definition}
Since $\vec{0} \in C_{p-1}(Y)$ trivially, it is immediate that any absolute $p$-cycle must be a relative $p$-cycle. On the other hand, if a relative $p$-cycle $\alpha$ differs from an absolute $p$-boundary $\d_{p+1}(\beta)$ by some (possibly trivial) $p$-chain $\gamma \in C_p(K_0)$, then $\alpha$ must also be a relative $p$-boundary. Since any absolute $p$-boundary may be written as $\alpha = \d_{p+1} (\beta) + \vec{0}$, then any absolute $p$-boundary must also be a relative $p$-boundary. The \textit{relative homology} of $(X, Y)$ is the quotient vector space $H_*(X, Y) = Z(X, Y) / B(X, Y)$.

\begin{definition}\label{def:relBoundingChain}
    A \textbf{relative bounding chain} for a chain $\xi \in C_p(X)$ is a chain $\beta \in C_{p+1}(X)$ such that $\xi - \partial_{p+1}(\beta) \in C_p(Y)$. In words, we say that $\xi$ bounds $\beta$ relative to $Y$.
\end{definition}
The following lemma of \cite{hatcher} summarizes the structure of relative homology.%, and we use it to establish our stability result in Theorem \ref{thm:stability}.
\begin{lemma}\label{lem:relChainMaps}
    If pairs of spaces $(X, Y)$ and $(K, L)$ satisfy $X \subseteq K$ and $Y \subseteq L$, then the chain map $g_p: C_p(X, Y) \rightarrow C_p(K, L)$ induces a homomorphism $g_p^* : H_p(X,Y) \rightarrow H_p(K,L)$ on relative homology. 
\end{lemma}

\subparagraph*{Persistence.}
Lemma \ref{lem:relChainMaps} indicates that the inclusion maps $G_tY \subseteq G_{t+1}Y$,  $F_tX \subseteq F_{t+1}X$, and $G_tY \subseteq F_{t}X$ yield induced maps on relative homology of the form 
$
H_p(F_1,G_1) \to H_p(F_2,G_2) \to \cdots \to H_p(F_N,G_N)
$. 
This is the (dimension $p$) \emph{persistent relative homology module}. Across all $p$, this gives rise to a graded persistent relative homology module $H_*(\FX, \subFX)$ which admits a unique barcode decomposition due to \cite{Gabriel}.

% ============================================

\section{U-match Decomposition}
\label{sec:umatch}

U-match decomposition is a form of matrix decomposition which was introduced to study persistent homology and its variants (persistent cohomology, extended persistence, etc.), both structurally and algorithmically \cite{umatch}. Every matrix admits a U-match decomposition, which can be computed in cubic time \cite{umatch}.

\begin{definition}\label{def:umatch}
    For a field $\field$, let $\field^{n \times m}$ denote the set of all $n \times m$ matrices with coefficients from $\field$. A \textbf{U-match decomposition} of a matrix $D \in \field^{n \times m}$ is given by a tuple of matrices $(\T,\M,D,\S)$ which satisfy the following three conditions: 
    \begin{itemize}
        \item $\T\M = D\S$.
        \item $\T \in \field^{n \times n}$ and $\S \in \field^{m \times m}$ are both upper-triangular with diagonal entries equal to one.
        \item $\M \in \field^{n \times m}$ is a \textit{generalized matching matrix}. Concretely, this means that $\M$ has at most one nonzero entry per row, and at most one nonzero entry per column.         
    \end{itemize}
\end{definition}

This decomposition is useful in a variety of contexts because the matrices $\S$ and $\T$ carry rich information about the kernels and images of several families of row and column submatrices of $\D$ which are relevant to the persistence computation.
U-match is a close relative of the well-known $R=DV$ decomposition, and there is a one-to-one correspondence between so-called \emph{proper} U-match decompositions and \emph{proper} $R=DV$ decompositions; see \cite{umatch} for details.

% ============================================

\section{Computing PRH via U-match}
\label{sec:umatchPRH}

Here we introduce the main result: a procedure to compute barcode decompositions for persistent relative homology, via matrix factorization. We also prove correctness and provide a cubic complexity bound. 
% -----------

\subsection{Main algorithm}
\label{sec:algorithmconventions}

Algorithm \ref{alg:prh} provides pseudocode for the main algorithm. In order to finesse certain edge cases, this procedure assumes, without loss of generality, that the  subspace $Y$ is equal to $X$; that is, $G_N = X = F_N$. In instances where this is not the case, one can always extend $\F$ and $\G$ via $G_{N+1}X = X = F_{N+1}X$. 

At a high level, the algorithm proceeds with a few key steps. 
First, a matrix $\D$ is constructed by permuting the rows of the boundary matrix $D$ in a manner compatible with filtration $\G$ (increasing values top to bottom), and permuting the columns in a manner compatible with the filtration $\F$ (increasing values left to right).
We then compute a U-match decomposition $\T\M = \D\S$. As we will later see, the columns of $\T$ and the columns of $\S$ contain a basis for the relative boundaries and a basis for the relative cycles, respectively. We then compute a second U-match decomposition, which yields a single matrix whose columns contain a basis for \emph{both} the relative cycle and the relative boundary spaces.  Both the barcode and the relative cycle representatives can be determined from matrices in this second U-match.

Two special functions are used in the pseudocode.  
The \emph{relative cycle sublevel set function} is the unique integer-valued function $f_{ker}$ defined on $C(X)$ such that $Z(F_tX, G_tX) = \{ c \in C(X) : f_{ker}(c) \le t \} $ for all $t$. That is, $c \in Z(F_tX, G_tX) \iff f_{ker}(c) \le t$. Intuitively, we think of $f_{ker}(c)$ as the filtration value where $c$ becomes a relative cycle, i.e., the filtration value where $[c]$ is born, as a relative homology class.
Similarly, the \emph{relative boundary sublevel set function} is the unique integer-valued function $f_{im}$ defined on $C(X)$ such that $B(F_tX, G_tX) = \{ c \in C(X) : f_{im}(c) \le t \} $ for all $t$. That is, $c \in B(F_tX, G_tX) \iff f_{im}(c) \le t$. 
Intuitively, we think of $f_{im}(c)$ as the filtration value where $c$ becomes a relative boundary, i.e., the filtration value where $[c]$ dies, as a relative homology class.

\begin{algorithm}[h!]
\caption{Persistent Relative Homology}

\label{alg:prh}
    \begin{algorithmic}[1]
        \Require {A finite cell complex $X$  equipped with filtrations $\F: F_1X \subseteq \cdots \subseteq F_NX$ and $\G: G_1 X \subseteq \cdots G_N X$. We require $G_NX = X = F_NX$ and $G_kX \subseteq F_kX$ for all $k$.}
        \Ensure {An interval decomposition of the persistent relative homology module $H_*(F_1X,G_1X) \to \cdots \to H_*(F_NX,G_NX)$. This decomposition is encoded as a set $I$ containing one triplet $(birth,death,representative)$ for each half-open interval $[birth,death)$ in the barcode, where $representative$ is a corresponding relative cycle representative.}
        \State{$\cg_1, \dots , \cg_m \gets$ a permutation of the cells of $X$ such that $b_F(\cg_1) \le \cdots \le b_F(\cg_m)$} \ 
        \State{$\cs_1, \dots, \cs_m \gets$ a permutation of the cells in $X$ such that $b_G(\cs_1) \leq \cdots \leq b_G(\cs_k)$}
        \State{$\D \gets$ the differential matrix for $\C(X)$, with rows and columns ordered such that $\partial(\cg_j) = \sum_i \D_{ij} \cs_i$. }
        \State{Define $\D_{\cs_i, \cg_j} = \D_{ij}$ (for all $i,j$;  only for downstream analysis). }
        \State{$(\T, \M, \D, \S) \gets $ a U-match decomposition of $\D$.}
        \State{Define $\T_{\cs_i \cs_j} = \T_{ij}$, and $\S_{\cg_i \cg_j} = \S_{ij}$ (for all $i,j$;  only for downstream analysis)}
        \State{$\cb_1, \dots, \cb_m \gets $ perm.\ of the cells of $X$ s.t.  $\rbb(COL_{\cb_1}(\T)) \le \cdots \le \rbb(COL_{\cb_m}(\T))$ }
        \State{$\cc_1, \dots, \cc_m \gets$ perm. of the cells of $X$ s.t. $\rcb(COL_{\cc_1}(\S)) \leq \cdots \leq \rcb(COL_{\cc_m}(\S))$}        
        \State{$\A \gets$ the permutation of $\T$ such that $\A_{ij} = \T_{\cs_i,\cb_j}$.}    
        \State{$\B \gets$ the permutation of $\S$ such that $\B_{ij} = \S_{\cs_i,\cc_j}$}
        \State{Define $\A_{\cs_i,\cb_j} = \T_{\cs_i,\cb_j}$ and $\B_{\cs_i,\cc_j} = \S_{\cs_i,\cc_j}$ (for all $i,j$;  only for downstream analysis)}
        \State{$(\TT, \MM, \A^{-1}\B, \SS) \gets$ a U-match decomposition of $\A^{-1}\B$.}
        \State{Define $\TT_{\cb_i \cb_j} = \TT_{ij}$, and $\SS_{\cc_i \cc_j} = \SS_{ij}$ (for all $i,j$;  only for downstream analysis)}
        \State{$I \gets \emptyset$}
        \For{each 
        %column $c = COL_j(\A \TT)$ of matrix $\A \TT$
        nonzero entry $\MM_{ij}$ of the matrix $\MM$
        }
            \State{$representative \gets \sum_k (\A \TT)_{ki} \cs_k$ }
            \State{$birth \gets f_{ker}(representative)$}
            \State{$death \gets f_{im}(representative)$}  
            \If{$birth \neq death$}
            \State{insert $(birth, death, representative)$ into $I$}       
            \EndIf         
        \EndFor
        \State{return $I$}
    \end{algorithmic}
\end{algorithm}

The algorithm produces several different, each of which can be indexed by either integers or by cells. (see Algorithm \ref{alg:prh} for details): 
\begin{equation}
        \A_{\cs_i \cb_j}
        = \A_{ij} 
        \hspace{0.45cm}
        \B_{\cs_i \cc_j} = \B_{ij}
        \hspace{0.45cm}    
        \M_{\cs_i \cg_j} = \M_{ij}
        \hspace{0.45cm}            
        \S_{\cg_i \cg_j} = \S_{ij}
        \hspace{0.45cm}
        \T_{\cs_i \cs_j} = \T_{ij}
        \hspace{0.45cm}        
        (\A\TT)_{\cs_i \cc_j} = (\A\TT)_{ij}
        \notag
\end{equation}

We will implicitly identify the columns of these matrices with the corresponding linear combination of cells. For example, 
$$
COL_{\eta}(\A) \equiv \sum_\sigma \A_{\sigma \eta} \cdot \sigma
$$
where $\sigma$ runs over all cells in $X$. Note that we do not have to place a subscript $i$ on the variable $\sigma$ in $\sum_{\sigma} \A_{\sigma \eta} \cdot \sigma$ because the cells themselves  are valid indices for the matrix $\A$.

% $\T_{ij} = \T_{\cs_i \cs_j}$, $\A_{ij} = \A_{\cs_i \cb_j} = \T_{\cs_i \cb_j}$, $\S_{ij} = \S_{\cg_i \cg_j}$, $\B_{ij} = \B_{\cs_i \cc_j} = \S_{\cs_i \cc_j}$, $(\A\TT)_{ij} = (\A\TT)_{\cs_i \cc_j}$

Note, further, that $\A$ and $\T$ are ``equivalent permutations,'' in the sense that $\A_{\eta \rho} = \T_{\eta \rho}$ for all cells $\eta, \rho$, although $\A_{ij} \neq \T_{ij}$ in general. Likewise $\B$ and $\S$ are equivalent permutations, in the sense that $\B_{\eta \rho} = \S_{\eta \rho}$ for all cells $\eta, \rho$, although $\B_{ij} \neq \S_{ij}$ in general. 

\subsection{Compatible ordered bases} 

An \emph{ordered basis} for a vector space $V$ is an ordered (nonrepeating) sequence of vectors $b = (b_1, \ldots, b_k)$ whose elements form a basis of $V$. 
We say that $b$ is \emph{compatible} with a filtration of subspaces $FV: F_1V \subseteq F_2V \subseteq \cdots \subseteq F_NV$ if, for each $k$, there exists an $\ell$ such that $(b_1, \ldots, b_\ell)$ forms a basis for the subspace $F_kV$. In this subsection we show that Algorithm \ref{alg:prh} produces compatible ordered bases for several several of the fundamental filtrations  in persistent relative homology. We will use these bases to prove correctness in a later section.

\begin{theorem}
    \label{thm:bcycles}
    The columns of matrix  $\S$ (respectively, matrix $\B$) contain a basis for each relative cycle space $Z(F_t X, G_t X)$.
    In fact, the sequence $COL_1(\B), \ldots, COL_m(\B)$ is a compatible ordered basis for the filtration $Z(F_1 X, G_1 X) \subseteq \cdots  \subseteq  Z(F_N X, G_N X)$.    
\end{theorem}
\begin{proof}
    The space $Z(F_t X, G_t X)$ is the set of all chains $c \in C(F_tX)$ such that $\partial c \in G_tX$.  This space can be written as the intersection of two subspaces:
        (A) the space of chains $C(F_tX)$, which corresponds to the first $\ell$ columns of $\D$, for some $\ell$, and
        (B) the space of inverse images $\partial^{-1}(G_tX)$, where $G_t X$ corresponds to the first $\ell$ rows of $\D$, for some $\ell$. 
    Therefore \cite[Theorem 10]{umatch} implies that the columns of $\S$ contain a basis for each subspace  $Z(F_t X, G_t X)$. The same conclusion holds for the columns of $\B$, because $\B$ is an equivalent permutation of $\S$ (c.f. Section \ref{sec:algorithmconventions}); in particular, the unordered set of chains represented by the columns of $\B$ equals the unordered set of chains represented by the columns of $\S$.    
    % by construction; in particular, because $\B$ is a permutation of $\S$, for all $j$ there exists a $k$ such that $\sum_i \B_{ij} \cs_i = \sum_i \S_{ik} \cg_i$.  
    Because the columns of $\B$ are arranged in increasing order with respect to the function $f_{ker}$, it follows that each $Z(F_t X, G_t X)$ is the span of the first  $\ell$ columns of $\B$, for some $\ell$.
\end{proof}

\begin{theorem}
    \label{thm:aboundaries}
    The columns of matrix  $\T$ (respectively, matrix $\A$) contain a basis for each relative boundary space $B(F_tX, G_t X)$. 
    In fact, the sequence $COL_1(\A), \ldots, COL_m(\A)$ is a compatible ordered basis for the filtration $B(F_1 X, G_1 X) \subseteq \cdots  \subseteq  B(F_N X, G_N X)$.    
\end{theorem}
\begin{proof}
    The space $B(F_tX, G_t X)$ is the sum of two subspaces: (A) the space of boundaries $\partial(F_tX)$, which corresponds to the first $\ell$ columns of $\D$, for some $\ell$, and (B) the space of chains $C(G_tX)$, where $G_t X$ corresponds to the first $\ell$ rows of $\D$, for some $\ell$. 
    Therefore \cite[Theorem 10]{umatch} implies that the columns of $\T$ contain a basis for each subspace $B(F_tX, G_t X)$.  
    The same conclusion holds for the columns of $\A$, because $\A$ is an equivalent permutation of $\T$ (c.f. Section \ref{sec:algorithmconventions}); in particular, the unordered set of chains represented by the columns of $\A$ equals the unordered set of chains represented by the columns of $\T$. 
    Because the columns of $\A$ are arranged in increasing order with respect to the function $f_{im}$, it follows that each $B(F_t X, G_t X)$ is the span of the first  $\ell$ columns of $\A$, for some $\ell$.
\end{proof}

\begin{theorem}
    \label{thm:atboundaries}
    The sequence $COL_1(\A\TT),\ldots,  COL_m(\A\TT)$ is a compatible ordered basis for the filtration $B(F_1 X, G_1 X) \subseteq \cdots  \subseteq  B(F_N X, G_N X)$.  
\end{theorem}
\begin{proof}
    This follows from Theorem \ref{thm:aboundaries}, because $\TT$ is invertible and upper triangular. In particular, each column $COL_k(\A\TT)$ is a nontrivial linear combination of the first $k$ columns of $\A$, so it belongs to the same filtration level as $COL_k(\A)$.
\end{proof}

Because $\MM = \TT^{-1}(\A^{-1}\B) \SS$ is a product of invertible matrices, $\MM$ too is invertible. Since $\MM$ is a generalized matching matrix, it follows that each row (respectively, column) of $\MM$ contains exactly one nonzero entry.
Let $\pi = (\pi_1, \pi_2, \ldots )$ be the unique permutation such that $\MM_{\pi_1, 1}, \MM_{\pi_2, 2}, \ldots$ are the nonzero entries of $\MM$.

\begin{theorem}
    \label{thm:atcycles}
    The sequence $COL_{\pi_1}(\A\TT), \ldots  COL_{\pi_m}(\A\TT)$ is a compatible ordered basis for the filtration $Z(F_1 X, G_1 X) \subseteq \cdots  \subseteq  Z(F_N X, G_N X)$.  
\end{theorem}
\begin{proof}
    Let $k$ be given. Because $\MM$ is a generalized matching matrix, for each $i$ there exists a nonzero scalar $\alpha$ such that $COL_{\pi_i}(\A \TT) = \alpha \cdot COL_{i}(\A \TT\MM)$. Therefore the span of the first $k$ vectors in the sequence ($COL_{\pi_1}(\A\TT), COL_{\pi_2}(\A\TT), \ldots$) equals the span of the first $k$ columns of $\A \TT \MM$. 
    Because $\TT \MM = (\A^{-1}\B) \SS$, the span of the first $k$ columns of $\A \TT\MM$ equals the span of the first $k$ columns of $\A (\A^{-1}\B) \SS = \B \SS$. Since $\SS$ is invertible and upper triangular, the span of the first $k$ columns of $\B \SS$ equals the span of the first $k$ columns of $\B$. Thus the span of the first $k$ vectors in  ($COL_{\pi_1}(\A\TT), COL_{\pi_2}(\A\TT), \ldots$)
    equals the span of the first $k$ columns of $B$. 
    The desired conclusion therefore follows from Theorem \ref{thm:bcycles}, which states that the columns of $\B$ form a compatible ordered basis for the filtration $Z(F_1 X, G_1 X) \subseteq \cdots  \subseteq  Z(F_N X, G_N X)$. 
\end{proof}

\subsection{Birth and death values}
\label{sec:birthdeathvalues}

Here, we provide the final detail required to implement Algorithm \ref{alg:prh}, by showing how to calculate $\rbb(COL_i(\A\TT))$ and $\rcb(COL_i(\A\TT))$ for any $i$. In fact, this calculation requires only a small number of lookup operations. We begin by reducing the problem to questions about $\S$ and $\T$.

\begin{theorem}
    We have $\rbb(COL_j(\A\TT)) = \rbb(COL_j(\A))=
    \rbb(COL_{\beta_j}(\T))$
    for all $j$. 
\end{theorem}
\begin{proof}
    The equation $\rbb(COL_j(\A\TT)) = \rbb(COL_j(\A))$ follows from Theorems \ref{thm:aboundaries} and \ref{thm:atboundaries}, which state that the sequences $(COL_1(\A), COL_2(\A), \ldots)$ and $(COL_1(\A\TT), COL_2(\A\TT), \ldots)$ are both compatible ordered bases for the filtration $B(F_1X,G_1X) \subseteq \cdots \subseteq B(F_NX,G_NX)$.
    The equation $\rbb(COL_j(\A))=
    \rbb(COL_{\beta_j}(\T))$ follows from the fact that column $_j$ of $\A$ represents the same chain as column $\beta_j$ of $\T$, by definition of $\A$.
\end{proof}

\begin{theorem} We have $    \rcb(COL_{\pi_i}(\A\TT))  = \rcb(COL_{\gamma_i}(\S))$ for all $i$.  
% (Recall $\pi$ is the unique permutation of $\{1, \ldots, m\}$ such that $\MM_{\pi_1,1}, \MM_{\pi_2,2}, \ldots$ are the nonzero entries of $\MM$.)
\end{theorem}
\begin{proof}
    Theorems \ref{thm:bcycles} and \ref{thm:atcycles} state that both the sequence  $(COL_{\pi_1}(\A\TT), COL_{\pi_2}(\A\TT), \ldots)$ and the sequence $(COL_1(\B), COL_2(\B), \ldots)$ are  compatible ordered bases for the filtration of cycles $Z(F_1 X, G_1 X) \subseteq \cdots  \subseteq  Z(F_N X, G_N X)$.  Therefore, $\rcb(COL_{\pi_i}(\A\TT)) = \rcb(COL_{i}(\B))$ for all $i$. The desired conclusion follows, because $COL_{i}(\B)$ and $COL_{\cc_i}(\S)$ represent the same chain in $C(X)$, by definition of $\B$.
\end{proof}

We have now reduced the calculation of birth and death values for columns of $\A \TT$ to columns of $\S$ and $\T$. These values can, in turn, be caculated with simple lookup formulae:

\begin{theorem}
    We have 
    $
    \rcb(COL_\alpha(\S)) 
    = 
    \max \big \{b_G(COL_{\alpha}(\M)), b_F(\alpha) \big \}
    $ for any cell $\alpha$.
\end{theorem}
\begin{proof}
    Let $\eta = COL_\alpha(\S)$. Recall that we implicitly identify this vector with the chain $\sum_{\rho} \S_{\rho \alpha} \cdot \rho$, where $\rho$ runs over all cells in $X$. The chain $\eta$ is a relative cycle in $F_tX/G_tX$ if and only if $\eta \in F_t X$ and $\partial(\eta) \in G_tX$.  Therefore $f_{ker}(\eta) = \max\{b_G(\partial \eta), b_F(\eta) \}$.
    
    Because $\S$ is upper triangular with 1's on the diagonal, the column $COL_\alpha(\S)$ is a sequence of scalars of form $(\S_{\sigma_1 \alpha}, \ldots, \S_{\sigma_k \alpha}, 0, \ldots, 0)$ where $\sigma_k = \alpha$ and $\S_{\sigma_k \alpha} =1$. Thus, in fact, $\eta \equiv \sum_{i \le k} \S_{\sigma_i \alpha} \sigma_i$.  Since we required $b_F(\sigma_1) \le \cdots \le b_F(\sigma_m)$ when the sequence $\sigma_1, \ldots, \sigma_m$ was constructed, it follows that $b_F(\eta) = b_F(\alpha)$.

    Now consider $\partial \eta = \D COL_\alpha(\S) = COL_\alpha(\D\S) = COL_\alpha(\T\M) = \T COL_\alpha(\M)$. If $COL_\alpha(\M) = 0$ then $\partial\eta = 0$, so $b_G(\partial\eta) =  b_G(0) = b_G(COL_\alpha(M))$. On the other hand, if $COL_\alpha(\M) \neq 0$ then there is a unique cell $\rho$ such that $\M_{\rho \alpha} \neq 0$. In this case $\T COL_\alpha(\M)$ is a nonzero scalar multiple of $COL_\rho(\T) = (\T_{\cs_1,\rho}, \ldots, \T_{\cs_k,\rho}, 0, \ldots, 0)$, where $\cs_k = \rho$ and $\T_{\cs_k,\rho}=1$. Since we required $b_G(\cs_1) \le \cdots \le b_G(\cs_m)$ when the permutation $\cs_1, \ldots, \cs_m$ was defined, it follows that $b_G(\partial \eta) = b_G(\rho) = b_G(COL_\alpha(M))$.

    Substituting the calculated values for $b_G(\partial \eta)$ and $b_F(\eta)$ into the expression $f_{ker}(\eta) = \max\{b_G(\partial \eta), b_F(\eta) \}$ yields the desired result.
\end{proof}

\begin{theorem}
    We have 
    \begin{equation}
    \begin{split}
    \rbb(COL_\alpha(\T)) 
    &= 
    \begin{cases}
        \min \{b_G(\alpha), b_F(\eta)\} & \text{if $\exists$ a (unique) cell } \eta \text{ such that } \M_{\alpha, \eta } \neq 0\\
        b_G(\alpha) & \text{otherwise }
    \end{cases}    
    \notag
    \end{split}        
    \end{equation}
\end{theorem}
\begin{proof}
    
    % -----------
    Because $b_G(\cs_1) \le \cdots \le b_G(\cs_m)$ by construction of $\cs_1, \ldots, \cs_m$, each subspace $G_tX$ can be expressed as the span of a set of form $\{\cs_1, \ldots, \cs_\ell\}$ for some $\ell$. Since $\T$ is upper triangular, the $\ell$th column of $\T$ represents a linear combination of cells $\cs_1, \ldots, \cs_j$; therefore $\{COL_1(\T), \ldots, COL_\ell(\T)\}$ has the same span as $\{\cs_1, \ldots, \cs_\ell\}$. It follows that (A) the columns of $\T$ contain a basis for each subpsace $G_tX$, and (B) the filtration value $b_G(COL_\ell(\T))$ equals the filtration value $b_G(\cs_\ell)$, for all $\ell$.

    Because $b_F(\cg_1) \le \cdots\le b_F(\cg_m)$ by construction, for each $t$ there exists an $\ell$ such that $\cg_1, \ldots, \cg_\ell$ spans $F_tX$. Since the columns of $\D$ are indexed by $\cg_1, \ldots, \cg_m$, it follows that the first $\ell$ columns of $\D$ span $\partial(F_tX)$. Because $\S$ is upper triangular, the span of the first $\ell$ columns of $D$ equals the span of the first $\ell$ columns of $\D\S = \T \M$. Because $M$ is a generalized matching matrix, each of these first $\ell$ columns of $\T\M$ is either zero or a scalar multiple of a (unique) column of $\T$. Therefore, the (linearly independent) columns of $\T$ contain a basis for $\partial(F_tX)$. It follows, moreover, that a  column $COL_i(\T)$ is a boundary in $\partial(F_tX)$ iff the first $\ell$ columns of $\T \M$ contain a nonzero scalar multiple of $COL_i(\T)$; this is equivalent to the condition that there exists a $j \le \ell$ such that $\M_{ij} \neq 0$. Recalling that $\M_{ij} = \M_{\cs_i, \cg_j}$ by definition, and noting that $\{\cg_1, \ldots, \cg_\ell\} = \{ \cg_p : b_F(\cg_p) \le t\}$, we see that this condition is equivalent to the existence of a cell $\cg_j$ such that $b_F(\cg_j) \le t$ and $\M_{i,j} \neq 0$. Thus, for each column $COL_i(\T)$ one of two conditions must hold: (A) row $i$ of $\M$ is zero, in which case $COL_i(T) \notin \partial(F_t X)$ for any $t$, or (B) there exists a $j$ such that $M_{ij} \neq 0$, in which case $COL_i(T) \in \partial(F_t) \iff b_F(\cg_j) \le t$.

    We have now shown that the columns of $\T$ contain both (A) a basis for $\partial(F_tX)$, for all $t$, and (B) a basis for $G_tX$, for all $t$. Therefore, the columns of $\T$ contain a basis for $B(F_tX, G_tX) = \partial(F_tX) + G_tX$ for all $t$. An elementary exercise therefore shows that $COL_i(\T) \in B(F_tX, G_tX)$ if and only if (I) $COL_i(\T) \in \partial(F_tX)$ or (II) $COL_i(\T) \in G_tX$. The filtration value $f_{im}(COL_i \T)$ is the minimum $t$ for which one of these two conditions holds. Condition (II) holds for all $t \ge b_G(COL_i(\T))$, and we have seen that $b_G(COL_i(\T)) = b_G(\tau_i)$. Condition (I) holds for all $t \ge t_0$, where $t_0$ is either $b_F(\cg_j)$ (if there exists $\cg_j$ such that $M_{\cs_i, \cg_j} \neq 0$), or $\infty$ (if row $j$ of $\M$ is zero). The desired conclusion follows if we substitute $\alpha = \cs_i$ and $\eta = \cg_j$.
\end{proof}

\subsection{Proof of correctness}
\label{sec:correctness}

\newcommand{\U}{U}
 
Let $\U$ denote the unordered set of columns of $\A\TT$, regarded as chains. Theorems  \ref{thm:atboundaries} and \ref{thm:atcycles} state that $\U$ contains a basis for the space of relative cycles $Z(F_tX, G_tX)$ and the space of relative boundaries $B(F_tX, G_tX)$, for each $t$. An elementary exercise in linear algebra therefore shows that the set difference $\Delta_t := \Big (\U \cap Z(F_tX, G_tX) \Big) \setminus \Big( U \cap  B(F_tX, G_tX)\Big)$ forms a basis for the quotient space $H_*(F_tX,G_tX) =  \frac{Z(F_tX, G_tX)}{B(F_tX, G_tX}$.  By definition of $f_{ker}$, an element $u \in \U$ belongs to $Z(F_tX, G_tX)$ iff $f_{ker}(u) \le t$. Likewise, $u$ belongs to $B(F_tX, G_tX)$ iff $f_{im}(u) \le t$. 
Thus $u$ belongs to the basis $\Delta_t$  if and only if  $f_{ker}(u) \le t < f_{im}(u)$.

These observations collectively show that the barcode of $(\F, \G)$ has one bar of form  $[f_{ker}(u),f_{im}(u))$ for each $u \in U$, and the homology class corresponding to this bar has relative cycle representative $u$. Thus, the set $I = \{ (f_{ker}(u), f_{im}(u), u): u \in U \}$ is a valid representation of the barcode, complete with representatives. This is the same set, $I$, returned by Algorithm \ref{alg:prh}. Thus the algorithm is correct.

\subsection{Complexity}

The algorithm comprises several lookup operations (linear time), sorting operations ($O(m \log m)$ time), and U-match factorizations ($O(m^3)$ time, c.f. \cite{umatch}), where $m$ is the number of cells in the cell complex. Therefore, Algorithm \ref{alg:prh} is $O(m^3)$.

% ============================================

\section{Performance optimization (for lag filtrations)}
\label{sec:lag}

A fundamental challenge in most persistent homology computations is the problem of factoring matrices which explode in size as a function of the input, often having hundreds of billions of rows and columns. The field of computational topology has made major strides in addressing this challenge in the special case where the matrix to be factored is $\D$, the differential of a cubical or simplicial complex. However, while many of the strides can be adapted to support the computation of the first U-match decomposition $(\T,\M,\D,\S)$ in Algorithm \ref{alg:prh}, they fail to assist with the second decomposition $(\TT,\MM,\A^{-1}\B, \SS)$.

This challenge naturally motivates the search for special cases where persistent relative homology calculations can be accelerated. Indeed, much of the existing literature in computational relative homology can be seen as an attempt to find specific data types where effective performance optimizations can be applied.

Here we present a new family of examples and corresponding U-match optimization: lag filtrations.
Let $X$ be a finite cell complex equipped with filtrations $\F: F_1X \subseteq \cdots \subseteq F_NX = X$ and $\G: G_1X \subseteq \cdots \subseteq G_NX = X$. We say that $\G$ is a \emph{lag filtration} of $\F$ if there is a constant $l$ such that $G_tX = F_{t-l}X$ for all $1 \le t \le N$, where by convention $F_tX := 0$ for $t \le 0$. In this case, there exists a U-match decomposition of form $(J,\M,\D,J)$, where $\D$ is a copy of the differential matrix with rows and columns sorted in (possibly non-unique) filtration order with respect to the birth function $b_F$ \cite[Remark 8]{umatch}. In fact, this decomposition can be computed using any algorithm that produces a valid $R=DV$ decomposition of $\D$ \cite[Section 1.3 and Lemma 16]{umatch}. 

The columns of the invertible upper triangular matrix $J$ contain a basis for every subspace of form (A) $F_tX$, (B) $\partial(F_tX)$, or (C) $\partial^{-1}(F_tX) = \{ c \in C(X): \partial c \in F_tX \}$; in addition, the columns of $J$ contain a basis for any subspace which can be obtained from subspaces of form (A), (B), or (C) via sum or intersection \cite[Theorem 25]{umatch}. In particular, the columns of $J$ contain a basis for every relative cycle space $Z(F_tX, G_tX) = F_tX \cap \partial^{-1}(G_tX) = F_tX \cap \partial^{-1}(F_{t-l}X)$ and every relative boundary space $G_tX + \partial(F_tX) = F_{t-l}X + \partial(F_tX)$.

This last observation implies that we can follow the exact procedure described in Section \ref{sec:correctness} to compute the persistent relative homology barcode of  $(\F,\G)$; one needs only to substitute the matrix $J$ for the matrix $\A\TT$. The values of the  relative boundary sublevel set function $f_{im}$ and the realtive cycle sublsevel set function $f_{ker}$ for each column $COL_j(J)$ can be computed in a manner similar to the one described in Section \ref{sec:birthdeathvalues}.

Taken together, these observations show that the persistent relative homology barcode can be computed at cost comparable to an $R=DV$ decomposition, in the special case of lag filtrations. While $R=DV$ has the same asymptotic complexity as Algorithm \ref{alg:prh}, in practice the performance is vastly improved due to highly optimized $R=DV$ solvers. We believe this to be one of many potential performance optimizations in this problem space.

% % ============================================

\section{Implementation}
\label{sec:implementation}

We provide an implementation of the optimized procedure for persistent relative homology barcodes, specialized for lag filtrations (c.f. Section \ref{sec:lag}), available at \cite{prh_code}. The implementation is an extension of Open Applied Topology \cite{OAT}, an open source computational topology package with a low-level backend written in Rust and a high-level frontend in Python.

\begin{figure}[ht]
    \centering
    \begin{subfigure}[b]{0.32\textwidth}
        \centering
        \includegraphics[width=\textwidth]{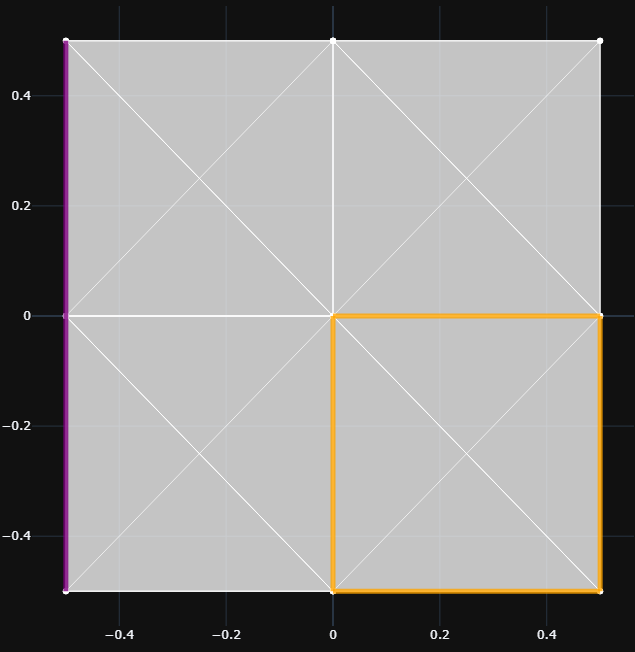}
        % \subcaption{A point cloud together with an absolute (orange) and relative (purple) cycle representative.}
        \label{fig:lag-simple-cloud}
    \end{subfigure}
    \hfill
    \begin{subfigure}[b]{0.32\textwidth}
        \centering
        \includegraphics[width=\textwidth]{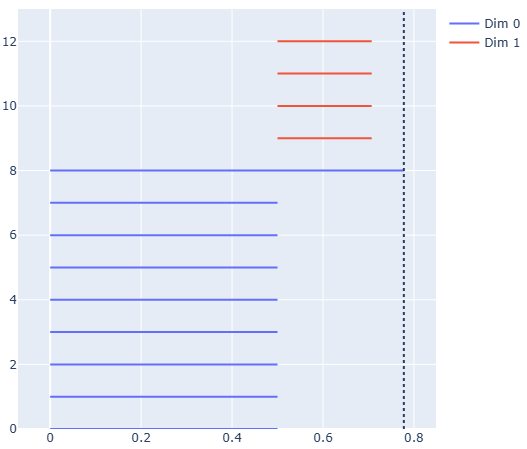}
        % \subcaption{Persistent homology barcode.}
        \label{fig:lag-simple-ph}
    \end{subfigure}
    \hfill
    \begin{subfigure}[b]{0.32\textwidth}
        \centering
        \includegraphics[width=\textwidth]{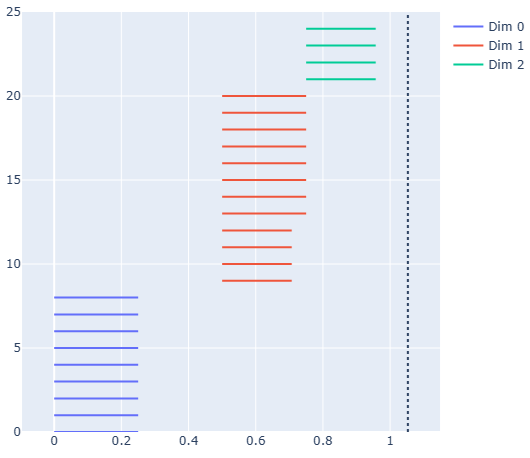}
        % \subcaption{Persistent relative homology barcode.}
        \label{fig:lag-simple-prh}
    \end{subfigure}
    \caption{A point cloud consisting of nodes on a $3 \times 3$ grid where adjacent nodes $a,b$ on the grid satisfy $\|a-b\|_2 = 1/2$ and diagonally adjacent nodes satisfy $\|a-b\|_2 = \sqrt{2}/2$. Included are examples of relative (purple) and absolute (orange) $1$-cycles \textbf{(left)} together with the standard persistent homology barcode \textbf{(center)} and persistent relative homology barcode \textbf{(right)} corresponding to a filtration with lag $l = 1/4$. Notice on the right figure that no class can have lifetime larger than $l$, including essential $0$-homology.}
    \label{fig:lag-simple}
\end{figure}

Figure \ref{fig:lag-simple} provides a conceptual illustration of a Vietoris-Rips complex built on a point cloud of grid nodes with lag $l = 1/4$. Notice there are four absolute $H_1$ classes whose interval modules are of the form $\mathbb{I} \langle 1/2, \sqrt{2}/2 \rangle$ (middle). There are eight relative $1$-homology classes whose interval modules are of the form $\mathbb{I} \langle 1/2, 3/4 \rangle$ (right). Four die at filtration value $\sqrt{2} / 2$ (length of the diagonal edges), while the others persist until all exterior edges enter the subcomplex (at 0.75). There are also four relative $2$-homology classes. The $2$-simplices emerging at filtration value $\sqrt{2} / 2$ are the surfaces of these spheres, while the $1$-simplices entering the subcomplex at $3/4$ are their boundaries. 

The example in Figure \ref{fig:lag-limit} computes  barcodes on a random point cloud in the Euclidean plane, for several distinct lag parameters $l$. As $l$ increases, the initial portion of the barcode converges to the absolute persistent homology barcode of the data. For smaller values of $l$, a variety of intermediate relative features emerge. We posit that these features may be of interest in scientific applications concerned with the local structure of data, e.g., materials science. 

\begin{figure}[ht]
    \centering
    % row 1
    \begin{subfigure}[b]{0.40\textwidth}
        \centering
        \includegraphics[width=\textwidth]{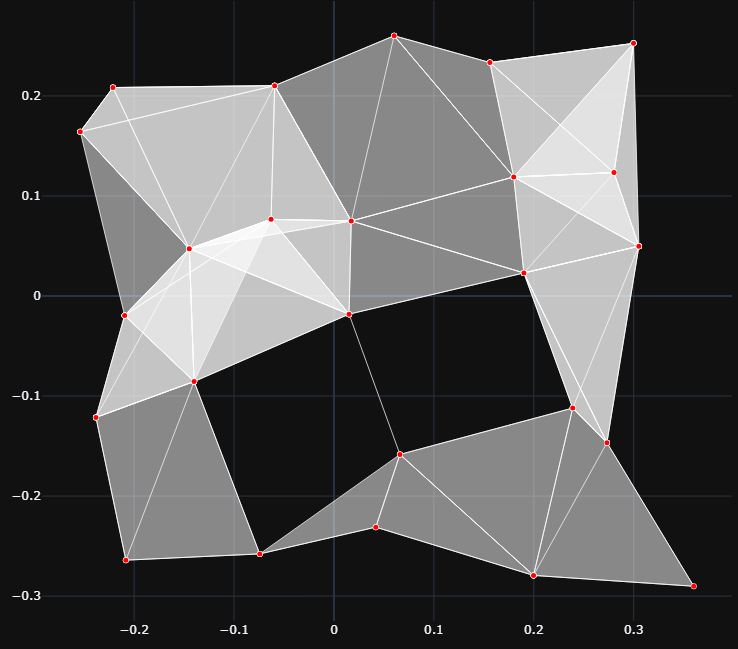}
        %\subcaption{}
        \label{fig:lag-cloud}
    \end{subfigure}
    \hfill
    \begin{subfigure}[b]{0.40\textwidth}
        \centering
        \includegraphics[width=\textwidth]{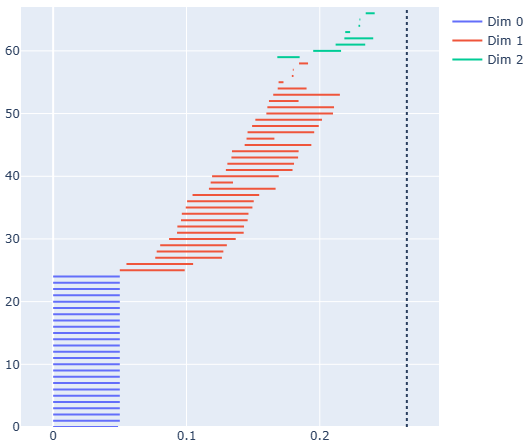}
        %\subcaption{}
        \label{fig:lag-005}
    \end{subfigure}

    \vskip\baselineskip % space between rows

    % row 2
    \begin{subfigure}[b]{0.40\textwidth}
        \centering
        \includegraphics[width=\textwidth]{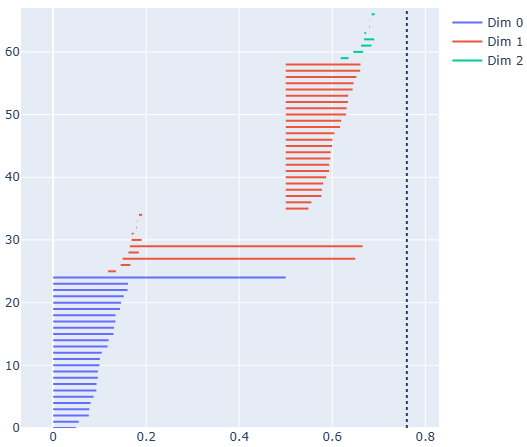}
        % \subcaption{}
        \label{fig:lag-050}
    \end{subfigure}
    \hfill
    \begin{subfigure}[b]{0.40\textwidth}
        \centering
        \includegraphics[width=\textwidth]{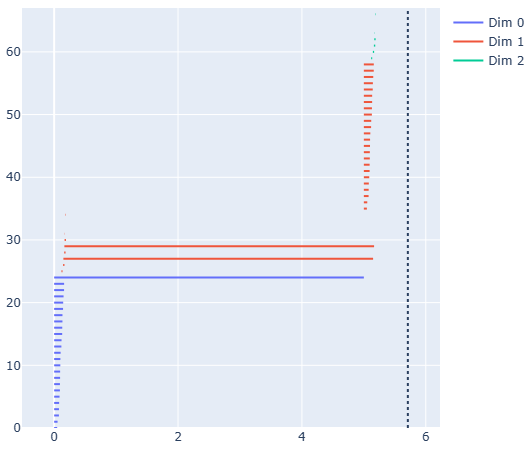}
        % \subcaption{}
        \label{fig:lag-500}
    \end{subfigure}

    \caption{A point cloud and (overlaid) Vietoris-Rips complex for fixed filtration value $\varepsilon$ \textbf{(top row, left)} together with a sequence of three persistent relative homology barcodes computed with lag parameters $\delta_1 = 0.05$ \textbf{(top row, right)}, $\delta_2 = 0.5$ \textbf{(bottom row, left)} and $\delta_3 = 5$ \textbf{(bottom row, right)}.}  
    \label{fig:lag-limit}
\end{figure}

% ============================================

% ============================================ 

\section{Conclusion}

We have provided two novel methods to compute persistent relative homology via U-match decomposition,  with an eye toward applications in TDA.  These methods can be analyzed in a highly transparent manner via elementary linear algebra. The main contribution of the present paper is an algorithm that works using (at most) two matrix decompositions. We compute not only a barcode for persistent relative homology but also persistent relative cycle representatives.
%We have shown that the resulting barcode is algorithmically stable via the Isometry Theorem. 
We also have provided a performance-optimized approach for the special case of lag-filtrations, and believe many more opportunities exist for performance enhancement, moving forward.

% ============================================

\newpage

\bibliography{prh-article}

\newpage

\appendix 

\end{document}